\documentclass{amsart}

\usepackage{amsmath,amsxtra,amssymb,amsfonts,latexsym, mathrsfs, dsfont,bm}
\usepackage{latexsym}
\usepackage{bbm}
\usepackage[usenames]{color}
\usepackage[hyperindex]{hyperref}
\usepackage[initials]{amsrefs}
\usepackage{tikz}
\usepackage{comment}
\usepackage{color}
\usepackage{graphics,epsf}         
\usepackage{enumerate}

\newtheorem{theorem}{Theorem}[section]
\newtheorem{lemma}[theorem]{Lemma}
\newtheorem{proposition}[theorem]{Proposition}

\newtheorem{corollary}[theorem]{Corollary}

\renewcommand{\q}{\quad}

\newcommand{\cn}{\mathcal N}

\newcommand{\ba}{\boldsymbol\alpha}
\newcommand{\bb}{\boldsymbol\beta}
\newcommand{\bep}{\boldsymbol\epsilon}

\DeclareMathOperator\supp{supp}

\newcommand{\mr}{\mathbb{R}}

\newcommand{\la}{\langle}
\newcommand{\ra}{\rangle}

\newcommand{\eqn}[1]{\begin{align}#1\end{align}}
\newcommand{\eq}[1]{\begin{align*}#1\end{align*}}
 
\newcommand{\bee}{\mathbf{e}}

\newcommand{\bz}{\mathbf{0}}
\newcommand{\bone}{\mathbf{1}}

\newcommand{\R}{{\mathbb R}}

\newtheorem*{theoremB}{Theorem B}

\begin{document}

\title[Multilinear Oscillatory integrals]{Higher decay inequalities for multilinear oscillatory integrals}

\author{Maxim Gilula}
\address{ Department of Mathematics, Michigan State University, Michigan, MI 48824, USA}
\email{gilulama@math.msu.edu}

\author{Philip T. Gressman}
\address{ Department of Mathematics, University of Pennsylvania, Philadelphia, PA 19104, USA}
\email{gressman@math.upenn.edu }

\author{Lechao Xiao}
\address{ Department of Mathematics, University of Pennsylvania, Philadelphia, PA 19104, USA}
\email{xle@math.upenn.edu}
\thanks{The second author is partially supported by NSF grant DMS-1361697 and an Alfred P. Sloan Research Fellowship.}
\date{\today}
\begin{abstract}
In this paper we establish sharp estimates (up to logarithmic losses) for the multilinear oscillatory integral operator studied by Phong, Stein, and Sturm \cite{PSS01} and Carbery and Wright \cite{CW02} on any product $\prod_{j=1}^d L^{p_j}(\R)$ with each $p_j \geq 2$, expanding the known results for this operator well outside the previous range $\sum_{j=1}^d p_j^{-1} = d-1$. Our theorem assumes second-order nondegeneracy condition of Varchenko type, and as a corollary reproduces Varchenko's theorem and implies Fourier decay estimates for measures of smooth density on degenerate hypersurfaces in $\R^d$.
\end{abstract}

\maketitle

\section{Introduction}\label{intro}

Let $\boldsymbol x$ $= (x_1,\dots, x_d)\in\mr^{d}$ and  let $\phi(\boldsymbol  x)$ be real analytic in some neighborhood of $\bz\in \mathbb R^d$. Fix a smooth cutoff function $\chi$ supported near the origin, and consider the multilinear functional 
\begin{align}\label{Multilinear}
\Lambda({\boldsymbol f})  = \int_{\mathbb R^d} e^{i\lambda \phi(\boldsymbol  x)} \chi(\boldsymbol  x)\prod_{j=1}^d f_j(x_j) d\boldsymbol  x,
\end{align}
where ${\boldsymbol f}= (f_1,\dots, f_d)$ is any $d$-tuple of locally integrable functions on $\mathbb R$. The purpose of this article is to study the asymptotic behavior in the real parameter $\lambda$ as $|\lambda| \rightarrow \infty$ of the norm of $\Lambda$ when viewed as a linear functional on $\prod_{j=1}^d L^{p_j}(\R)$.

Bilinear variants of this form have a long history in harmonic analysis in connection with the study of Fourier integral operators and Radon-like transforms (see, e.g., Greenleaf and Uhlmann \cite{GU90} and Seeger \cite{SEE99}). In the 1990s, Phong and Stein initiated the study of these oscillatory integrals as a subject in its own right \cite{PS92}. Their program focused primarily on $L^2 \times L^2$ estimates of weighted and unweighted varieties \cites{PS94-1,PS94-2,PS97}, as this particular case was most directly connected to the earlier FIO roots. In this setting, the undamped bilinear case with real analytic phase was ultimately settled in \cite{PS97}, with the transition to $C^\infty$ phases being later accomplished by Rychkov \cite{RY01} and Greenblatt \cite{GR05}. These works demonstrated the primary role of the (reduced) Newton polyhedron of the phase $\phi$, which had also been identified as a key object in Varchenko's study of scalar oscillatory integrals some twenty years earlier \cite{VAR76}. To define the Newton polyhedron, first expand $\phi(\boldsymbol  x)=\sum_{\ba}c_{\ba}\boldsymbol  x^{\ba}$ near the origin, where $\boldsymbol  x^{\ba}  = x_1^{\alpha_1}\cdots  x_d^{\alpha_d}$, and define the Taylor support of $\phi$ by $\supp(\phi)=\{\ba : c_{\ba}\neq 0\}.$ 
Let $\mr_\ge$ denote the nonnegative real numbers. The Newton polyhedron of $\phi$, denoted by $\mathcal N(\phi)$, is defined to be the convex hull of
\eq{\bigcup_{\ba\in \supp(\phi)} \ba+\mr_{\ge}^d,}
and the Newton distance $d_\phi$ of $\phi$ is defined to be the minimum over all $t$ such that $(t,\dots,t)\in \mathcal N(\phi).$ In the specific case of the form \eqref{Multilinear}, modulating each function $f_j$ by a function of the form $e^{-i \lambda \phi_j(x_j)}$, it can be easily seen that terms in the power series of $\phi$ which depend on only one coordinate function do not affect the norm of $\Lambda$ on $\prod_{j=1}^d L^{p_j}({\mathbb R})$, so it will be assumed without loss of generality that every $\ba \in \supp(\phi)$ has at least two strictly positive components. After removing all single-variable terms in the Taylor support of $\phi$, the resulting Newton polyhedron corresponds to the object known as the the reduced Newton polyhedron in other contexts.

The success of the program of Phong and Stein to establish $L^2 \times L^2$ estimates for \eqref{Multilinear} prompted generalizations and extensions to a variety of higher-dimensional settings, including results of Carbery, Christ, and Wright \cite{CCW99} as well as Carbery and Wright \cite{CW02}. The most natural extension of the work of Phong and Stein to higher dimensions turned out to be \eqref{Multilinear} itself, which was studied by Phong, Stein, and Sturm \cite{PSS01} as well as Carbery and Wright \cite{CW02}. The main theorem of Phong, Stein, and Sturm which is most closely related to the present work is as follows:

\begin{theoremB}[\cite{PSS01}]
Let $\alpha^{(1)},\ldots,\alpha^{(K)} \in {\mathbb N}^d \setminus \{0\}$ be $K$ given vertices, and let $S \in {\mathbb R}[x_1,\ldots,x_d]$ be any polynomial of degree $n_S$. Set
\[ D(\alpha^{(1)},\ldots,\alpha^{(K)}) = \left\{ x \in U: \ \ |S^{(\alpha^{(k)})}(x)| > 1, \ \ 1 \leq k \leq K \right\}. \]
Let $N^*(\alpha^{(1)},\ldots,\alpha^{(K)})$ be the reduced Newton polyhedron generated by the vertices $\alpha^{(k)}$, i.e., the Newton polyhedron generated by those vertices $\alpha^{(k)}$ with at least two strictly positive components. Then for any algebraic domain $D \subset D(\alpha^{(1)},\ldots,\alpha^{(K)})$ and any $\alpha \in N^*(\alpha^{(1)},\ldots,\alpha^{(K)})$, we have
\begin{equation} |\Lambda ({\boldsymbol f})| \leq C |\lambda|^{-\frac{1}{|\alpha|}} \ln^{d-\frac{1}{2}} (2 + |\lambda|) \prod_{j=1}^d \|f_j\|_{p_j}, \ \ d \geq 2. \label{pssineq} \end{equation}
Here $\lambda\neq 0$ is any number, and $\frac{1}{p_j'} = 1 - \frac{1}{p_j} = \frac{\alpha_j}{|\alpha|}$. The constant $C$ depends only on $n_S$, $|\alpha|$, and the type $r,n,d,w$ of $D$.
\end{theoremB}

Phong, Stein, and Sturm's purpose in proving Theorem B was to establish a robust stability result for the multilinear form \eqref{Multilinear}, focusing on the role of the Newton polyhedron. In the present paper, we focus on a somewhat different question following naturally from Theorem B, namely, the possible range of exponents $p_j$. Generally speaking, the cases of most interest will occur when the exponents $p_j$ are large, since in the opposite extreme, when $p_j = 1$ for one or more indicies $j$, standard scaling arguments reduce the question of boundedness to a uniform estimate {\it a la } Phong, Stein, and Sturm which must be valid over a family of multilinear forms exhibiting a lower degree of multilinearity.

In the large-exponent regime we study here, the decay in $\lambda$ of the form \eqref{Multilinear} is generally of a higher order than in the inequality \eqref{pssineq}. This extra decay brings with it additional difficulties not encountered in \cite{PSS01}, which make it necessary to introduce certain auxiliary assumptions not found there. The formulation we have chosen is essentially a higher-order version of the so-called Varchenko hypothesis \cite{VAR76}.
Let $\mathcal F(\phi)$ to denote the set of compact faces of $\mathcal N(\phi)$. In particular, the set of zero dimensional faces $\mathcal V(\phi) \subset \mathcal F(\phi)$ is the collection of vertices of $\mathcal N(\phi)$. The Varchenko-type nondegeneracy condition we assume is: for all $F\in \mathcal F(\phi)$, 
\begin{align}\label{v-condition}
\bigcap_{ i \neq j } \{\boldsymbol  x: \partial_i\partial_j\phi_F(\boldsymbol  x)=0\} \subset \bigcup_{ 1\leq j \leq d} \{\boldsymbol  x: x_j =0\}.
\end{align}
In words, we assume that any point at which all off-diagonal terms of the Hessian matrix $\nabla^2 \phi_F$ simultaneously vanish for all faces $F$ must belong to a coordinate hyperplane.
Our principal result is
\begin{theorem}\label{main}
Let $\phi$ be real analytic and satisfying \eqref{v-condition}. Let $p_j\in [2,\infty]$ for $1\leq j \leq d$. 
If the support of $\chi$ is contained in a sufficiently small neighborhood of $\bz$, then for any real number $\nu > 2$,
\begin{align}\label{maineq}
|\Lambda(\boldsymbol  f)| \lesssim |\lambda|^{-\frac 1 {\nu}} \log^{m}(2+|\lambda|) \prod_{j=1}^d \|f_j\|_{p_j}
\end{align}
for some implicit constant independent of $\boldsymbol  f$ and some $m\geq 0$ if and only if 
\begin{align}
\frac{\boldsymbol \nu}{\boldsymbol p'}=\Bigg(\frac{\nu}{p_1'},\dots, \frac{\nu}{p_d'}\Bigg) \in \mathcal N(\phi),
\end{align}
where $p'$ denotes the conjugate of $p$. 
\end{theorem}
For readers interested in the exponent of the logarithmic factor, our proof provides a value of $m$ which can be calculated easily from the geometry of the Newton polyhedron: $m = 0$ if $\frac{\boldsymbol \nu}{\boldsymbol p'}$ is an interior point of $\mathcal N(\phi)$, and $m = d-\ell$ if the face of lowest dimension containing $\frac{\boldsymbol \nu}{\boldsymbol p'}$ itself has dimension $(\ell-1)$. 
The value of $m$ may not be sharp in general, but is sharp when all $p_j =\infty$. In particular, Theorem \ref{main} recovers the classical result of Varchenko under the somewhat milder hypothesis \eqref{v-condition}, where in this case $\boldsymbol\nu\in\cn(\phi)$:
\begin{corollary}
Let $\phi$ be the same as above and let $(\ell-1)$ denote the smallest dimension over all faces of $\mathcal N(\phi)$ containing $\boldsymbol \nu$, where $\nu=d_\phi$ is the Newton distance of $\phi$. Then   
$$
|\Lambda({\boldsymbol f})| \lesssim |\lambda|^{-\frac 1 {\nu}} \log^{d-\ell} (2+|\lambda|)\prod_{j=1}^d \|f_j\|_{\infty}.
$$
When $\chi (\bz) \neq 0$, the power of the log term is also sharp. 
\end{corollary}
The usefulness of this functional variant of Varchenko's theorem is even more apparent when the functions $f_j$ are taken to be complex exponentials. 
Taking $f_j(x_j) = e^{i \xi_j x_j}$ and setting $\lambda = \xi_{d+1}$, the above corollary together with standard nonstationary phase estimates implies sharp estimates for the Fourier decay of measures of smooth density on the surface $({\boldsymbol x}, \phi({\boldsymbol x}))$; that is, for $\boldsymbol \xi\in\mr^{d+1},$ 
$$
\Bigg|\int e^{i\boldsymbol \xi\cdot ({\boldsymbol x}, \phi({\boldsymbol x}))}  \chi({\boldsymbol x})  d {\boldsymbol x}\Bigg| \lesssim \|\boldsymbol \xi\|_2^{-\frac 1 {d_\phi}}\log^{d-\ell}(2+|\xi_{d+1}|).  
$$

One can compare our methods for proving Lemma \ref{LemmaKey} with the more typical resolution of singularities methods of, for example, Greenblatt \cite{GRB10, GRB12}, Collins, Greenleaf, and Pramanik \cite{CGP13}, or Xiao \cite{X2013}.  Lemma \ref{LemmaKey} allows us to avoid more algebraic considerations by carefully studying the various nonisotropic scalings which make the Taylor polynomials associated to compact faces of the Newton polyhedron homogeneous. 

%
%
%
%
%
%
%
%
%
%
%
%
%
%
%

%
%
%
%
%
%
%
%
%
%
%
%
%
%
%

\section{Notation and preliminaries}\label{Layout}
Given two vectors $\boldsymbol  x=(x_1,\dots, x_d), \boldsymbol  y=(y_1,\dots,y_d) \in\mr^d,$ and a scalar $c>0$, we define
\begin{itemize}
	\item $\boldsymbol  x\boldsymbol  y= (x_1y_1,\dots, x_dy_d),$

	\item $\boldsymbol  x ^{\boldsymbol  y}= x_1^{y_1}\cdots x_d^{y_d},$
	\item $c^{\boldsymbol  x}=(c^{x_1},\dots, c^{x_d}),$ 
	\item $\boldsymbol  c= (c,\dots, c)$, 
	\item $\frac{\boldsymbol  x}{\boldsymbol  y}=\Big(\frac{x_1}{y_1},\dots, \frac{x_d}{y_d}\Big)$ if each $y_i\neq 0$, and 
	\item  $\partial_{\boldsymbol x}^{\ba} = \prod_{j=1}^d \partial_{x_j}^{\alpha_j}=\prod_{j=1}^d \partial_{j}^{\alpha_j}$.
\end{itemize}
Let $\bep = (\epsilon_1,\dots, \epsilon_d)$, where each $\epsilon_j$ is a positive dyadic number. We use $Q_{\bep}$ to denote the box
\[
Q_{\bep} = \prod_{j=1}^d [\epsilon_j, 8 \epsilon_j].
\]
To fully exploit the nondegeneracy condition \eqref{v-condition}, we define the following quantity:   
for any subset $S$ of $\mathbb R^d$, let  
\eq{
\| \phi\|_{V(S)} =\inf_{\boldsymbol  x\in S} \max_{i\neq j} | x_i x_j \partial_i\partial_j \phi(\boldsymbol  x)|.
}
The following lemma is key to establishing Theorem \ref{main}. It asserts that under the nondegeneracy condition \eqref{v-condition}, one can control the absolute value of some mixed partials of $\phi$ from below in each box $Q_{\bep}$. The proof can be found in Section \ref{POL}.    
\begin{lemma}\label{LemmaKey}
Let $\phi$ be real analytic, satisfying \eqref{v-condition}. Then there is a neighborhood $U$ of $\bz$ and a positive constant $K$
 such that for all $Q_{\bep}\subset U$
 \begin{align}
 \| \phi\|_{V(Q_{\bep})} \geq K \bep ^{\ba},\quad {\rm for \,\, all} \quad  \ba\in \mathcal V(\phi).
 \end{align}
\end{lemma}
We also need control of the absolute value of mixed derivatives of $\phi $ from above: 
\begin{lemma}\label{LemmaUpper}
There is a neighborhood $U$ of $\bz$ and a constant $K'$ such that for all ${\boldsymbol a }\in \{0, 1, 2, 3\}^d$ and all $Q_{\bep}\subset U$, 
\begin{align}
\sup_{\boldsymbol  x\in Q_{\bep}}|{\boldsymbol x} ^{\boldsymbol a}  \partial_{\boldsymbol x}^{\boldsymbol a} \phi({\boldsymbol  x}) | \leq 
K' \max_{\ba\in \mathcal V(\phi)} \bep^{\ba}.
\end{align} 
\end{lemma}
The proof of this lemma follows directly from the analyticity of the function $\phi$ and the fact that all $\ba$ in the Taylor support of $\phi$ are either convex combinations of verticies or correspond to terms of higher order than any such convex combination.  
Now let $U$ be as in Lemma \ref{LemmaKey} and Lemma \ref{LemmaUpper}. 
By coupling these two lemmas with the Mean Value Theorem, one can prove a slightly stronger version of Lemma \ref{LemmaKey}: 
\begin{corollary}\label{SubDecomposition}
There exists $N\in\mathbb N$ depending on $K$, $K'$ and $\phi$, such that the following holds. Each $Q_{\bep}\subset U$ can be dyadically decomposed into a collection of  
$2^{dN}$ congruent boxes $Q_{\bep, l}$ for $ {1\leq l\leq 2^{dN}} $ such that for each $Q_{\bep, l}$ and for all $\ba\in \mathcal V(\phi)$, there is a fixed pair $(i,j)$ such that 
  \begin{align}\label{est_KK}
 \inf_{\boldsymbol  x\in  2 Q_{\bep,l}} | x_i x_j \partial_i\partial_j \phi(\boldsymbol  x)| \geq  \frac{K}{2} \bep ^{\ba}.
 \end{align}
 \end{corollary}
 We only outline the proof of this corollary here and leave the details to the interested reader. Decompose $Q_{\bep}$ into $2^{dN}$ congruent boxes $Q_{\bep, l}$ of dimensions $2^{-N}\bep$, with $N$ to be determined momentarily.  For each $l$, by Lemma \ref{LemmaKey}, there exist  ${\boldsymbol x}\in Q_{\bep, l}$ and a pair $(i,j)$, such that $|x_ix_j\partial_{i}\partial_{j}\phi({\boldsymbol x})| \geq 
K \bep ^{\ba}$, for all  $ \ba\in \mathcal V(\phi)$. By choosing $N$ large enough (independent of $\bep$), \eqref{est_KK} is a consequence of this estimate, Lemma \ref{LemmaUpper}, line integrals, and the Mean Value Theorem.  
 
The main analytic tool to be employed is the following operator van der Corput lemma due to Phong and Stein \cite{PS94-1, PS94-2}. The proof can be found throughout the literature;  see, for example, \cite{GR04}.
 \begin{lemma}\label{OS}
Let $\chi(x,y)$ be a smooth function supported in a box with dimensions $\delta_1 \times \delta_2$ such that   
$|\partial_y^l \chi(x,y)| \leq C_1 \delta_2^{-l} 
$
for $l =0,1, 2$ and some $C_1>0$. 
Let $\mu>0$ and $S(x,y)$ be a smooth function s.t. for all $(x,y)$ in the support of $\chi,$
\begin{align*}
C_2 \mu \leq |\partial_x\partial_yS(x,y)| \leq C_3 \mu
\q\q\mbox{and}\q\q
 |\partial_x\partial_y^l S(x,y)| \leq C_3 \frac{\mu}{\delta_2^{l}} \q\mbox{for}\q l=1,2\,.
\end{align*}
Then the operator defined by 
\begin{align*}
T_\lambda f(x)  =\int_{-\infty}^\infty e^{i\lambda S(x,y)} \chi (x,y)f(y)dy 
\end{align*}
satisfies
\begin{align}\label{PS2}
\|T_\lambda f\|_2 \leq C |\lambda\mu|^{-\frac 1 2 }\|f\|_2,
\end{align}
where the constant $C$ depends on $C_1, C_2$, and $C_3$, but is independent of $\mu$, $\lambda$ and other information of the phase $S$. 
\end{lemma}
The rest of this paper is organized as follows. Lemma \ref{LemmaKey} is proved in Section \ref{POL}. Estimates of $\Lambda (\boldsymbol f)$ are established first for a single box in Section \ref{POM}. 
The main tool used to sum over all boxes is Lemma \ref{linear}, which is proved in Section \ref{LP}. The main theorem will be established in Section \ref{FR}.

%
%
%
%
%
%
%
%
%
%
%
%
%
%
%

\section{Proof of lemma \ref{LemmaKey}}\label{POL}
The methods below are very similar to those in the first author's thesis \cite{GIL16}. For the rest of the section, we write $\phi=P_m+R_m,$ where $P_m$ is the Taylor polynomial at the origin of $\phi$ of degree at most $m$ and $R_m$ is the remainder. The integer $m$ is chosen so that $\mathcal N(\phi)=\mathcal N(P_m);$ $m$ always exists because the Newton polyhedron has finitely many extreme points. Write $P_m(\boldsymbol  x)=\sum_{|\ba|\le m} c_{\ba}\boldsymbol  x^{\ba}$ and $R_m(\boldsymbol  x)=\sum_{|\ba|= m} h_{\ba}(\boldsymbol  x)\boldsymbol  x^{\ba}$ for some real analytic functions $h_{\ba}$. For each $1\le i\neq j\le d$ we can write $x_ix_j\partial_i\partial_j\phi(\boldsymbol  x)$ as
\eqn{
	& x_ix_j\partial_i\partial_j\phi(\boldsymbol  x)= \sum_{|\ba|\le m} c_{\ba}' \boldsymbol  x^{\ba} + \sum_{|\ba|= m} h_{\ba}'(\boldsymbol  x) \boldsymbol  x^{\ba},\label{phij}}
where $c_{\ba}'=\alpha_i\alpha_jc_{\ba},$ and $h_{\ba}'$ depends on $i$ and $j$. For each compact $F\subset \mathcal N(\phi),$ we can write the polynomial $x_ix_j\partial_i\partial_j P_m(\boldsymbol  x)$ in (\ref{phij}) as
\eqn{\sum_{\ba\in F}c_{\ba}'\boldsymbol  x^{\ba} + \sum_{\ba\notin F}c_{\ba}'\boldsymbol  x^{\ba}.\label{splitsum}}
The goal is to show for all $\boldsymbol  x$ small enough, one may choose $F$ and $1\le i,j\le d$ so that \eqref{phij} is dominated by the sum over $\alpha \in F$ and all remaining terms are of a perturbative quality. 

\subsection{Supporting hyperplanes of $\mathcal N(\phi)$ and scaling}
The difficulty of dividing the sum \eqref{phij} into finitely many terms on a compact face $F$ of $N(\phi)$ and a remainder term of sufficiently small magnitude comes when on some box $Q_{\bep}$, there are many choices of $\ba \in N(\phi)$ such that $\bep^{\ba}$ are of roughly the same magnitude as $\bep^{\bb}=\max_{\ba\in \mathcal N(\phi)}\bep^{\ba}$ which do not belong to the face $F$ itself. To simplify the computations, we will associate faces $F$ with affine hypersurfaces $H$ in the natural way, i.e.,  $F$ and $H$ are associated to one another when $F = H \cap N(\phi)$. In this hyperplane language, the following proposition describes how to adjust a candidate ``dominant hyperplane'' when there are corresponding remainder terms which are not small enough for our purposes. 
Below, one should think of $\ba^1,\dots, \ba^n$ as vertices of a compact face of a Newton polyhedron with $\bep^{\ba^\ell}$ very close or equal to $\bep^{\bb}=\max_{\ba\in \mathcal N(\phi)}\bep^{\ba}$ in the sense  that there is some $K>0$ such that $K\bep^{\bb}\le \bep^{\ba^\ell}\le \bep^{\bb}$. For instance, the hyperplane under consideration could contain $\bb, \ba^1,\dots, \ba^{n-1},$ but not $\ba^{n}.$ In this case we can move to a hyperplane that contains all $(n+1)$ points, at the cost of moving to a larger face $F' \supseteq F$ (and later analyzing $\phi$ over a bigger box than $Q_\epsilon$ and possibly changing the indices $i$ and $j$).

	\begin{proposition}\label{scaleprop}
		Let $\bep\in (0,1)^d$ and let $\ba^1,\dotsc, \ba^{n}\in\mr^d$ be linearly independent. Assume for all $1\le k \le n$ there is a positive $K<1$ such that $K\bep^{\bb}\le \bep^{\ba^k}\le \bep^{\bb}$ for some multiindex $\bb$. There is some $b\in (0,1)$ depending only on $\ba^1,\dotsc, \ba^n$ and $K$ such that for some $\boldsymbol  y\in [b,b^{-1}]^d$ and all $1\le k \le n$ we have
		\eqn{\boldsymbol  y^{\ba^k}=\bep^{\ba^k-\bb}.\label{eqlts}}
		Moreover, if $\ba=\sum_k \lambda_k \ba^k$ for $\sum_k \lambda_k=1,$ then 
		\eqn{\boldsymbol  y^{\ba}=\bep^{\ba-\bb}. \label{scaleprop2}}
	\end{proposition}
	
	\begin{proof}
		Let $A$ be the $n\times d$ matrix with rows $\ba^1,\dotsc,\ba^{n}$. 		
		Without loss of generality, assume that the first $n$ columns of $A$ are linearly independent. Let $\boldsymbol  v\in\mr^n$ be the vector defined componentwise by $v_k=\log_2(\bep^{\ba^k-\bb})$. Consider the equation $\tilde{A}\boldsymbol  u=\boldsymbol  v$, where $\tilde{A}=(\alpha_i^j)_{1\le i,j\le n}.$ Since $\tilde{A}$ has full rank, we can invert $\tilde A$ and write $\boldsymbol  u=\tilde{A}^{-1}\boldsymbol  v$. Fixing $\rho=\|\tilde{A}^{-1}\|_{\infty \rightarrow \infty},$ we have that
		$$\|\boldsymbol  u\|_\infty\le \rho \|\boldsymbol  v\|_\infty.$$
		Therefore $-\|\boldsymbol  v\|_\infty \rho\le u_k \le \|\boldsymbol  v\|_\infty \rho$ for all $k.$ However, $\|\boldsymbol v\|_\infty \leq |\log K|$, so 
		$$K^{\rho}\le 2^{-\|\boldsymbol v\|_\infty \rho}  \le 2^{u_k} \le 2^{\|\boldsymbol v\|_\infty \rho} \le K^{-\rho }.$$
		Hence, letting $b=K^{\rho}\in (0,1)$, we see that the vector $\boldsymbol  y\in [b,b^{-1}]^d$ defined by $y_k=2^{u_k}$ for $1\le k \le n$ and $y_k=1$ otherwise satisfies the system of equations (\ref{eqlts}). Finally, (\ref{scaleprop2}) follows from rewriting $\ba-\bb=\sum_k \lambda_k(\ba^k-\bb).$
	\end{proof}
	The last part of the proposition will be useful shortly because $F\cap \supp(P_m)$ might not be a linearly independent set, but it is always contained in the affine hull of $(\dim(F)+1)$ many linearly independent vectors contained in $F.$

\subsection{The main result required for Lemma \ref{LemmaKey}}

Suppose $K_0,\ldots,K_{d-1} \in (0,1)$ are fixed constants (to be determined later) and any $d$-tuple of dyadic numbers $\bep$. We will call 
$\bep$ $n$-dominated when there is a $n$-dimensional compact face $F \subset {\mathcal N}(P_m)$ such that for all $\bb$ and $\bb'$ in $F$ and all $\ba \in {\mathcal N}(P_m) \setminus F$,
\[ \bep^{\bb} = \bep^{\bb'} \mbox{ and } \bep^{\ba} \leq K_n \bep^{\bb}. \] 
The property of $n$-domination is extremely useful because when estimating \eqref{splitsum} on the box $Q_{\bep}$, it allows a favorable pointwise upper bound of those terms not belonging to the face $F$. Unfortunately, not all dyadic $d$-tuples $\bep$ are $n$-dominated for some $n$. However, by Proposition \ref{scaleprop}, if for any given $\bep$, there is an  $n$-dimensional face $F$ in ${\mathcal N}(P_m)$ of ``dominant terms,'' i.e., such that $\bep^{\bb} = \bep^{\bb'}$ for all $\bb$ and $\bb'$ in $F$ and $\bep^{\ba} \leq \bep^{\bb}$ for all $\ba \not \in F$, then either $\bep$ is $n$-dominated, or the nearby dyadic $d$-tuple $\boldsymbol y^{-1} \bep$ has an even higher-dimensional face of dominant terms (in applying Proposition \ref{scaleprop}, simply fix $\bb$ to be any vertex in ${\mathcal V}(P_m)$ which maximizes $\bep^{\bb}$ and take $\ba_1,\ldots,\ba_n$ to be all other vertices which maximize $\bb \mapsto \bep^{\bb}$ together with any other vertices which, while not maximizing $\bb \mapsto \bep^{\bb}$ are in the forbidden range $\bep^{\ba} \geq K_n \bep^{\bb}$). Finally, note that when the face $F$ of dominant terms associated to an $\bep$ is $(d-1)$-dimensional, then automatically $\bep$ is $(d-1)$-dominated when $\bep$ is sufficiently small because there are only finitely many one-dimensional curves in $\R_{\geq}^d$ on which $\bep$ must lie to be $(d-1)$-dominated (i.e., curves defined by those nonisotropic scalings of $\R^d$ which make $\phi_F$ homogeneous for some $(d-1)$-dimensional face $F$). Therefore, by induction, we can say that for any dyadic $d$-tuple $\bep$ which is sufficiently small (depending on $K_0,\ldots,K_{d-1}$), there is some $\bep' \in (0,1)^d$ such that $\bep'$ is $n$-dominated for some $n$ and $\bep' \bep^{-1}$ has components bounded above and below by constants depending only on $\mathcal N(P_m)$ and $K_0,\ldots,K_{n-1}$.

We can now finish the proof of the main lemma. Consider once again the sums\eqref{phij} and \eqref{splitsum}. Fix any dyadic $d$-tuple $\bep$ and let $\bep'$ be the $n$-dominated $d$-tuple identified above which is close to $\bep$.  Let $\bb$ be any vertex in the dominant face $F$ associated to $\bep'$. If we define coordinates $\boldsymbol z \in \R^d$ so that $\boldsymbol x = \bep' \boldsymbol z$ for all $\boldsymbol x \in Q_{\bep}$, then
\[ \left| x_i x_j \partial_i \partial_j \phi(\boldsymbol x) - \bep'^{\bb} \sum_{\ba \in F} c'_{\ba} {\boldsymbol z}^{\ba} \right| \leq K_n {\bep'}^{\bb} \sum_{\ba \not \in F} \left|c_{\ba}' {\boldsymbol z}^{\ba} \right| + C \sum_{|\ba| = m} \left| \bep'^{\ba} {\boldsymbol z}^{\ba} \right| \]
where the constant $C$ depends only on the functions $h_{\ba}$. If we assume that $\bep$ is sufficiently small (or equivalently, that the cutoff function $\chi$ of \eqref{Multilinear} is supported sufficiently near the origin) depending on $K_0,\ldots,K_{n-1}$ and $\phi$, we may assume that
\[ C \sum_{|\ba| = m} \left| \bep'^{\ba} {\boldsymbol z}^{\ba} \right| \leq \frac{1}{3} \max_{i \neq j} \left| \bep'^{\bb} \sum_{\ba \in F} c'_{\ba} {\boldsymbol z}^{\ba} \right| \]
for every $\boldsymbol z$ such that $\bep' \boldsymbol z \in Q_{\bep}$ since by induction the coordinates of $\boldsymbol z$ are bounded away from $0$ and $\infty$, which means by the nondegeneracy hypothesis \eqref{v-condition} that
\[ \max_{i \neq j} \left| \sum_{\ba \in F} c'_{\ba} {\boldsymbol z}^{\ba} \right| \]
is bounded below uniformly in $\bep$ and $\boldsymbol z$ by a constant that depends only on $\phi$ and $K_0,\ldots,K_{n-1}$. Likewise, if $K_n$ is chosen sufficiently small depending on $K_0,\ldots,K_{n-1}$ and $\phi$, we may also assume that
\[ K_n \sum_{\ba \not \in F} \left|c_{\ba}' {\boldsymbol z}^{\ba} \right| \leq \frac{1}{3} \max_{i \neq j} \left| \sum_{\ba \in F} c'_{\ba} {\boldsymbol z}^{\ba} \right|, \] 
which finally implies that
\[ \max_{i \neq j} |x_i x_j \partial_i \partial_j \phi(\boldsymbol x)| \geq \frac{1}{3} \bep^{\bb} \left| \sum_{\ba \in F} c'_{\ba} {\boldsymbol z}^{\ba} \right| \geq C_n \bep^{\bb} \]
uniformly for all $\boldsymbol x \in Q_{\bep}$ with some constant that depends on $\phi$ as well as the choice of $K_0,\ldots,K_n$. Since $\bep'^\bb$ dominates $\bep^{\ba}$ for all $\ba \in {\mathcal V}(\phi)$ and the coordinates of $\bep'^{-1} \bep$ are bounded above and below, Lemma \ref{LemmaKey} follows.

\section{Single box estimates}\label{POM}
 Let $Q_{\bep}$ be a fixed box and $\chi_\epsilon$ be a smooth function supported in $Q_{\bep}$ with
 \begin{align}\label{ED01}
  | \partial_{\boldsymbol x}^{\boldsymbol k} \chi_{\bep}({\boldsymbol x})| \leq C_{\boldsymbol k} {\bep} ^{- \boldsymbol k},  \quad \forall \,\,{\boldsymbol k} \in \mathbb N^d   .
 \end{align}

The goal of this section is to obtain estimates of the multilinear form $\Lambda (\boldsymbol f)$ when the support of $\boldsymbol f$ is restricted to a single box $Q_{\bep}$ in the first orthant. This is achieved by coupling Corollary \ref{SubDecomposition} with Lemma \ref{OS}; this goal is stated below:  

 \begin{lemma}\label{single}
 Let $Q_{\bep}$ as in Corollary \ref{SubDecomposition}. For all $(p_1,\dots, p_d) \in [2,\infty]^d$, there holds 
 \begin{align}\label{Penn}
|\Lambda(\boldsymbol  f \chi_{{{\bep}}})| 
\lesssim \min_{\ba\in\mathcal N (\phi)}\{|\lambda {\bep}^{\ba} |^{-\frac 1 2}  {\bep}^{\frac 1 {\boldsymbol p'}}, \quad  {\bep}^{\frac 1 {\boldsymbol p'}}\}
  \prod_{1\leq j \leq d}\|f_j\|_{p_j}. 
\end{align}
\end{lemma}
 Let $\bb\in \mathcal V(\phi)$ be such that ${\bep}^{\bb} \geq {\bep}^{\ba}$ for all $\ba\in \mathcal N(\phi)$.
 It suffices to show  
 \begin{align}\label{Penn01}
|\Lambda(\boldsymbol  f \chi_{{{\bep}}})| 
\lesssim \min\{|\lambda {\bep}^{\bb} |^{-\frac 1 2}  {\bep}^{\frac 1 {\boldsymbol p'}}, \quad  {\bep}^{\frac 1 {\boldsymbol p'}}\}
  \prod_{1\leq j \leq d}\|f_j\|_{p_j}. 
\end{align}
To prove this, 
let $\{Q_{{\bep}, l}\}_{1\leq l \leq 2^{Nd}}$ be the corresponding sub-decomposition as in Corollary \ref{SubDecomposition}.
 Apply again a smooth partition to $\chi_{\bep}$ and write it as the sum of $\chi_{{\bep},l}$, 
with each $\chi_{{\bep},l}$ supported in $2Q_{\bep,l}$ and satisfying estimates similar to \eqref{ED01}.    
For each $Q_{\bep,l}$, there is a fixed pair $(i,j)$ such that 
  \begin{align*}
 \inf_{\boldsymbol  x\in  2 Q_{\bep,l}} | x_i x_j \partial_i\partial_j \phi(\boldsymbol  x)| \geq \frac {K}{2} {\bep} ^{\bb} . 
 \end{align*}
Notice also for all ${\boldsymbol x}\in   2Q_{\bep,l}$ and for $b =1, 2, 3$, 
\begin{align*}
 | x_i x_j^{b} \partial_i\partial_j^b \phi({\boldsymbol x})| \leq   {K_2} {\bep} ^{\bb}
\end{align*} 
for some constant $K_2$ independent of ${\bep}$. In the above two estimates, we have applied Lemma \ref{LemmaUpper} and Corollary \ref{SubDecomposition}.
Using Fubini's Theorem,  
\begin{align*}
\Lambda(\boldsymbol  f \chi_{{{\bep},l}})  = \int_{\mathbb R^{d-2}}\left(\int_{\mathbb R^2} e^{i\lambda \phi({\boldsymbol x})}f_i(x_i)f_j(x_j) \chi_{{{\bep},l}}({\boldsymbol x})dx_i dx_j \right) \prod_{k \neq i, j}f_k(x_k) dx_k.
\end{align*} 
Applying Lemma \ref{OS} 
to the inner integral and then $L^1$ norms to other functions yields 
\begin{align}\label{Decay}
|\Lambda(\boldsymbol  f \chi_{{{\bep},l}})| 
\lesssim  |\lambda {\bep}^{\bb} {\epsilon}_i^{-1}{\epsilon}_j^{-1}|^{-\frac 1 2} \|f_i\|_2\|f_j\|_2\prod_{k\neq i,  j}\|f_k\|_1.
\end{align} 
 H\"older's inequality and the lower bounds $p_j\geq 2$ give 
\begin{align*}
|\Lambda(\boldsymbol  f \chi_{{{\bep},l}})| 
\lesssim  |\lambda {\bep}^{\bb} {{\epsilon}}_i^{-1}{{\epsilon}}_j^{-1}|^{-\frac 1 2} |{{\epsilon}}_i|^{\frac 1 2 -\frac 1 {p_i}}\|f_i\|_{p_i}
|{{\epsilon}}_i|^{\frac 1 2 -\frac 1 {p_j}}\|f_j\|_{p_j}\prod_{k\neq i,  j}|{{\epsilon}}_k|^{ 1  -\frac 1 {p_k}}\|f_k\|_{p_k},
\end{align*} 
that is,
\begin{align} \label{Est01}
|\Lambda(\boldsymbol  f \chi_{{{\bep},l}})| 
\lesssim  |\lambda {\bep}^{\bb} |^{-\frac 1 2} \prod_{1\leq k \leq d}|{{\epsilon}}_k|^{ 1  -\frac 1 {p_k}}\|f_k\|_{p_k}.
\end{align}
The first estimate in (\ref{Penn01}) follows by summing over $1\leq l \leq 2^{dN}$. 
Inserting the absolute value into the integral $|\Lambda(\boldsymbol  f \chi_{{{\bep}}})| $ and employing H\"older's inequality yields 
\begin{align}\label{Est02}
|\Lambda(\boldsymbol  f \chi_{{{\bep}}})| 
\lesssim  \prod_{1\leq j \leq d}|{{\epsilon}}_j|^{1-\frac 1 {p_j}}\|f_j\|_{p_j},
\end{align}
which is the second estimate of \eqref{Penn01}.

%
%
%
%
%
%
%
%
%
%
%
%
%
%
%

\section{Summing over all boxes}\label{LP}

To estimate the multilinear form $\Lambda (\boldsymbol f)$, we need to sum \eqref{Penn} over all $Q_{\bep}$,  
which is achieved by the following lemma: 
\begin{lemma}\label{linear}
Let $\boldsymbol z\in \mr^d$ have positive components and assume $F\subset \cn(\phi)$ of dimension $(\ell-1)$ is the lowest dimensional face containing $\boldsymbol  \gamma=\nu \boldsymbol z$ for some unique $\nu >2.$ Then the following is true for $\lambda\ge 2$: 
$$\sum_{j_1,\dots, j_d=0}^\infty \min_{\substack{N\in \{0, 1/2\}, \\\ba\in \mathcal N(\phi)}}\{\lambda^{-N}2^{\la N\ba-\boldsymbol z, \boldsymbol  j\ra }\}\lesssim \lambda^{-\frac{1}{\nu}}\log^{d-\ell}(\lambda).$$
\end{lemma}
Choosing $N=0$ we see that for any index $j_i$,
$$\sum_{j_1=0}^\infty\cdots\sum_{j_i=\log(\lambda)/\gamma_i}^\infty\cdots \sum_{j_d=0}^\infty 2^{-\la\boldsymbol z, \boldsymbol  j\ra}\lesssim \lambda^{-\frac{z_i}{\gamma_i}}=\lambda^{-\frac{1}{\nu}}.$$
Hence, it is enough to obtain the bound 
	\eqn{\sum_{j_1=0}^{\log(\lambda)/\gamma_1}\cdots\sum_{j_d=0}^{\log(\lambda)/\gamma_d}  \min_{\substack{N\in \{0, 1/2\}, \\\ba\in \mathcal N(\phi)}}\{\lambda^{-N}2^{\la N\ba-\boldsymbol z, \boldsymbol  j\ra}\}\lesssim \lambda^{-\frac{1}{\nu}}\log^{d-\ell}(\lambda)\label{varsum}} 
Here it is more natural to work in a continuous setting: the sum in (\ref{varsum}) may be bounded above by a uniform constant times
\eqn{\int_{0}^{\log(\lambda)/\gamma_1}\cdots \int_{0}^{\log(\lambda)/\gamma_d}  \min_{\substack{N\in \{0, 1/2\}, \\\ba\in \mathcal N(\phi)}}\{\lambda^{-N}e^{\la  N\ba -\boldsymbol z, \boldsymbol  x\ra}\} d\boldsymbol  x. \label{varint}}
Since $ F\ni \boldsymbol  \gamma$ is dimension $(\ell-1)$ and is not contained in a coordinate hyperplane, there are linearly independent $\ba^1,\dots, \ba^\ell\in F$ whose convex hull contains $\boldsymbol  \gamma,$ so write 
\eqn{\boldsymbol  \gamma= \sum_{i=1}^\ell \lambda_i\ba^i.\label{convcomb}}
For $1\le i\le \ell$ let $\theta_i=2\frac{\lambda_i}{\nu}$ and $\theta_0= 1-\frac{2}{\nu}$. All $\theta_i$ are positive and their sum is $1$ by the restriction placed on $\nu$. Moreover, we can check
\eqn{\theta_0(-\boldsymbol z) +\sum_{i=1}^\ell \theta_i \Big(\frac{\ba^i}{2}-\boldsymbol z\Big)=\bz.\label{convo}}

The integral (\ref{varint}) can be bounded above by 
\eqn{\int_{0}^{\log(\lambda)/\gamma_1}\dotsi \int_{0}^{\log(\lambda)/\gamma_d} \min_{1\le i\le \ell} \{e^{-\la \boldsymbol z, \boldsymbol  x\ra}, \lambda^{-\frac{1}{2}}e^{\la \frac{\ba^i}{2} -\boldsymbol z, \boldsymbol  x\ra}\} d\boldsymbol  x.\label{varint2}}
Without loss of generality we may assume $\ba^1,\dots, \ba^\ell,\bee_{\ell+1}, \dots, \bee_d$ are linearly independent and define the invertible matrix $A$ by
$$A\ba^i = \bee_i \text { for } 1\le i \le \ell,$$
and
$$A \bee_i = \bee_i \text{ for } \ell< i\le d.$$
Note that $\la A^T \boldsymbol  x, \ba^i\ra  = x_i$ for $1\le i\le \ell.$ Let $R=\{\boldsymbol y\in\mr^d: 0\le \la A^T \boldsymbol  y,\bee_j\ra\le \log(\lambda)/\gamma_j\text{ for all } 1\le j\le d\}$. Applying the change of variables $\boldsymbol  x=A^T \boldsymbol  y,$ up to a factor depending only on $A$ the integral (\ref{varint2}) equals
\eqn{\int_{R} \min_{1\le i\le \ell}\{e^{-\la A\boldsymbol z, \boldsymbol  y\ra}, \lambda^{-\frac{1}{2}} e^{\la A(\frac{\ba^i}{2}-\boldsymbol z), \boldsymbol  y\ra}\} d\boldsymbol  y.\label{varint3}}
First integrating over directions $\ell<i\le d$,
$$\int_{\substack{ 0\le \la A^T \boldsymbol  y,\bee_i\ra\le \log(\lambda)/\gamma_i\\ \ell<i\le d}} dy_{\ell+1}\dotsm dy_d \lesssim \log^{d-\ell}(\lambda).$$
We can therefore bound (\ref{varint3}) above by 
\eqn{\log^{d-\ell}(\lambda) \int_{\mr^\ell} \Big{|}\min_{1\le i\le \ell}\{e^{-\la A\boldsymbol z, \boldsymbol  y\ra}, \lambda^{-\frac{1}{2}} e^{\la A(\frac{\ba^i}{2}-\boldsymbol z), \boldsymbol  y\ra}\}\Big{|} dy_1\dotsm dy_\ell.\label{varint4}}
Since $A\ba^i=\bee_i$ and $\sum_{i=1}^\ell \lambda_i =1,$ we see
$$\la A\boldsymbol z ,\log(\lambda)\bone\ra = \frac{1}{\nu} \log(\lambda)\sum_{i=1}^\ell \lambda_i \la A\ba^i, \bone\ra = \frac{1}{\nu} \log(\lambda).$$
Exponentiating, we obtain $e^{-\la A\boldsymbol z,\log(\lambda)\bone\ra}=\lambda^{-\frac{1}{\nu}}.$ This calculation inspires the change of variables $y\to y+\log(\lambda)\bone$, after which we factor out $\lambda^{-\frac{1}{\nu}}$ and bound (\ref{varint4}) above by $\lambda^{-\frac{1}{\nu}} \log^{d-\ell}(\lambda)$ times
$$\int_{\mr^\ell} \Big{|}\min_{1\le i\le \ell}\{e^{-\la A\boldsymbol z, \boldsymbol  y\ra}, e^{\la A(\frac{\ba^i}{2}-\boldsymbol z), \boldsymbol  y\ra}\}\Big{|} dy_1\dotsm dy_\ell.$$
Note the factor $\lambda^{-\frac{1}{2}}e^{\frac{\log(\lambda)}{2}}=1$ from the change of variables above. By (\ref{convo}), 
$$\mathbf{0} = -\theta_0A\boldsymbol z+\sum_{i=1}^\ell \theta_i A\Big(\frac{\ba^i}{2}-\boldsymbol z\Big).$$
Since $A(\frac{\ba^i}{2}-\boldsymbol z)$ for $1\le i\le \ell$ are linearly independent, 
$$\sup_{\|\boldsymbol  y\|_2=1}\min_{1\le i\le \ell} \{-\la A\boldsymbol z, \boldsymbol  y\ra, \la A\Big(\frac{\ba^i}{2}-\boldsymbol z\Big), \boldsymbol  y\ra\}<0.$$
By homogeneity, there is a constant $c=c(\ba^1,\dots, \ba^\ell, \boldsymbol z)>0$ such that
$$\min_{1\le i\le \ell} \{-\la A\boldsymbol z, \boldsymbol  y\ra, \la A\Big(\frac{\ba^i}{2}-\boldsymbol z\Big), \boldsymbol  y\ra\}<-c \|\boldsymbol  y\|_2.$$
After a polar change of variables, we can bound (\ref{varint4}) by a constant independent of $\lambda$ times
$$\lambda^{-\frac{1}{\nu}}\log^{d-\ell}(\lambda) \int_0^\infty e^{-cr}dr \lesssim \lambda^{-\frac{1}{\nu}}\log^{d-\ell}(\lambda).$$


\section{Proof of the main Theorem}\label{FR}
Theorem \ref{main} will be a direct consequence of Lemma \ref{single} and Lemma \ref{linear}. 
Write 
 $$
 \chi({\boldsymbol x}) = \sum_{sign\in\{+,-\}^d} \chi_{sign}(\boldsymbol x) \quad {\rm for \,\, all} \prod_{1\leq j\leq d}x_j \neq 0, 
 $$
where $\chi_{sign}$ is the restriction of $\chi$ to the orthant corresponding to $sign\in\{+,-\}^d$. By the triangle inequality, it suffices to prove \eqref{maineq} for each $\chi_{sign}$. Without loss of generality, we will prove the case when $\chi$ is restricted to the first orthant. Let $\chi^+ =\chi_{sign}$ with $sign =+^d$. 
 By applying a smooth partition, one can write   
 $$
 \chi^+({\boldsymbol x}) = \sum_{\bep} \chi_{\bep}(\boldsymbol x),  
 $$
where each $\chi_{\bep} $ is a smooth function supported in a dyadic box $Q_{\bep}$ with 
 \begin{align}\label{ED01}
| \partial_{\boldsymbol x}^{\boldsymbol k} \chi_{\bep}({\boldsymbol x})| \leq C_{\boldsymbol k} {\bep} ^{- \boldsymbol k},  \quad \forall \,\,{\boldsymbol k} \in \mathbb N^d   .
 \end{align}
Let $\boldsymbol p\in [2,\infty]^d$. By the triangle inequality and Lemma \ref{single},  
$|\Lambda(\boldsymbol  f \chi^+)| $ can be controlled above by a positive constant times
\[
 \sum_{\bep= 2^{-\boldsymbol j}, \,\,\boldsymbol j \in \mathbb N^d}\,\,\,\min_{\ba\in\mathcal N (\phi)}\{|\lambda {\bep}^{\ba} |^{-\frac 1 2}  {\bep}^{\frac \bone {\boldsymbol p'}}, \quad  {\bep}^{\frac \bone {\boldsymbol p'}}\}
  \prod_{1\leq k \leq d}\|f_k\|_{p_k}. 
\]
Note that all components of $\frac \bone {\boldsymbol p'}$ are positive. If $\nu>2$ is such that $\nu\frac{\bone}{\boldsymbol p'}=\frac{\boldsymbol\nu }{\boldsymbol p'}\in \mathcal N(\phi)$, then \eqref{maineq} follows from the above estimate and Lemma \ref{linear}. In particular, when $\nu>2$ one can take  $m=0$ in \eqref{maineq} if $\frac{\boldsymbol\nu}{\boldsymbol p'}$ is an interior point of $\mathcal N(\phi)$ and $m= d-\ell$ if the face of lowest dimension containing $\frac{\boldsymbol\nu}{\boldsymbol p'}$ itself has dimension $(\ell-1)$. 

It remains to show that the estimate \eqref{maineq} is sharp up to a logarithmic factor.  For convenience, let us define the dual polyhedron of $\mathcal N(\phi)$ to be the subset of $\R^d_{\geq}$ such that 
	\begin{align}
	\mathcal N(\phi)^* = \{ \boldsymbol w \in \mathbb R_\ge^d: \la \ba, \boldsymbol w \ra \geq 1,\,\,  \forall \,\,\ba\in \mathcal N(\phi)\}.
	\end{align}
By a similar construction, the double dual $\mathcal N(\phi)^{**}$ can easily be checked to equal $\mathcal N(\phi)$. 
Likewise, it is not difficult to see that there is a constant $\delta >0$, depending on $\phi$ but independent of $\lambda,$ such that 
$|\lambda  \phi(\delta|\lambda|^{-\boldsymbol w})| \leq 10^{-10}$  for all $\boldsymbol w\in \mathcal N(\phi)^*$ and all $|\lambda|$ sufficiently large. If $\boldsymbol f$ is the characteristic function of the box
$|x_j| \leq \delta |\lambda| ^{-w_j}$ for $ 1\leq j \leq d$, then   
$$
|\Lambda (\boldsymbol  f) | \sim \|\boldsymbol  f\|_1  = 2^d\delta^d |\lambda|^{-\la \bone, \boldsymbol w\ra } \sim |\lambda|^{-\la \bone,\boldsymbol w\ra} .
$$
Then the estimate \eqref{maineq} implies 
$$
|\lambda|^{-\la \bone,\boldsymbol w \ra} \lesssim \log^m (2+|\lambda|)|\lambda| ^{-\frac 1 \nu} |\lambda|^{-\la \frac{\bone}{\boldsymbol  p}, \boldsymbol w\ra}.
$$
Letting $|\lambda| \to \infty$ implies  
$$
\Big\la \frac{\boldsymbol \nu}{\boldsymbol p'}, \boldsymbol w \Big\ra \geq 1 
$$
for all $\boldsymbol w \in  \mathcal N(\phi)^*$. Consequently,  $ \frac{ \boldsymbol \nu }{\boldsymbol p'} \in \mathcal N(\phi)^{**} =\mathcal N(\phi)$.  

\bibliographystyle{plain}

\end{document}